\documentclass[12pt]{amsart}

\usepackage{amssymb}
\usepackage{url}
\usepackage{color}

\usepackage{hyperref}
\hypersetup{
    colorlinks=true,
    linkcolor=blue,
    filecolor=magenta,      
    urlcolor=cyan,
    pdftitle={Sharelatex Example},
    bookmarks=true,
    pdfpagemode=FullScreen,
    citecolor=teal,
    }

\usepackage[colorinlistoftodos]{todonotes}

\makeatletter
\providecommand\@dotsep{5}
\makeatother

\usepackage[latin1]{inputenc}

\newtheorem{thm}{Theorem}

\newtheorem{lem}[thm]{Lemma}
\newtheorem{cor}[thm]{Corollary}

\theoremstyle{definition}

\newcommand{\R}{\mathbb{R}}
\newcommand{\N}{\mathbb{N}}

\newcommand{\g}{\tilde{g} }
\newcommand{\m}{\mathcal}
\newcommand{\M}{\mathfrak{M}}

\DeclareMathOperator{\supp}{supp}

\begin{document}

\title{On finite   quotient Aubry set  for generic   geodesic flows}
 
\author{Gonzalo Contreras}
\address{CIMAT \\
          A.P. 402, 36.000 \\
          Guanajuato. GTO \\
          M\'exico.}
\email{gonzalo@cimat.mx}

\author{Jos\'e  Ant\^onio G. Miranda}
\address{Universidade Federal de Minas Gerais,
Av. Ant\^onio Carlos 6627, 
31270-901, Belo Horizonte,
MG, Brasil.}
\email{jan@mat.ufmg.br}

\subjclass[2000]{ 37C40, 37C50, 37C99}

\begin{abstract}
We study  the structure of the Mather and  Aubry sets  for  the family  of lagrangians   given by  the kinetic energy associated to a riemannian metric $ g$  on a closed manifold $ M$. 
In this case the Euler-Lagrange flow is the geodesic flow of $( M,g)$. 
We prove that there exists a residual subset $ \m G$  of the set  of all conformal  metrics to g, such that, if $ \overline g \in \m G$ then the corresponding   geodesic flow has a finitely many ergodic c-minimizing measures, for each non-trivial cohomology class $ c \in H^1(M,\mathbb{R})$. 
This implies that,  for any $ c \in H^1(M,\mathbb{R})$,  the quotient Aubry set for the cohomology class c  has a finite number of elements for this particular family of lagrangian systems. 
 \end{abstract}

\maketitle 

\section{Introduction}

Let $ M $ be a closed smooth manifold  
endowed with a  $C^\infty$-riemannian metric $ g $ on $ M $ and let 
  $L: TM \to \R $ be the  lagrangian  defined as the kinetic  energy  corresponding to g, i.e, $$ L(x,v) = \frac{1}{2} g_x(v,v) .$$ 
Then the  Euler-Lagrange equation for this  particular class of Tonelli lagrangian, that will be called {\it riemannian lagrangian},  coincides with the  equation of the  geodesics on  $(M, g)$:   
\begin{equation}
\label{geodesic eq.}
\frac{D}{dt}\dot{\gamma}(t)=0.
\end{equation}
Therefore the corresponding Euler-Lagrange flow   $\phi_t: TM\to TM  $ is  the geodesic flow of the riemannian manifold $(M, g)$.

\smallskip 

In this paper, we  study  generic   properties of  $c$-minimizing measures for the class of riemannian lagrangians by   performing conformal perturbations of the metric $g$, that  do not depend of the  non-trivial cohomology class $ c \in H^1(M,\R)$. 
Recall that a metric $ \tilde g $ is  conformally equivalent  to $ g $,  if there exist a  smooth function $ f: M \to \R $,  such that $ \tilde g_x(\cdot,\cdot ) = e^{f(x)} g_x(\cdot,\cdot) $ for all $ x \in M$.

\bigskip

The concept of minimizing measures was introduced 
by  Mather in \cite{Mather:1991} for time periodic Tonelli lagrangian 
(see section~\ref{Review of the Mather-Aubry Theory}). In this setting,
 in  \cite{Mane:1996a},  Mañé proved   the existence of  a residual set $ \m O(c)$  in $ C^\infty (M) $ (that depends on $ c \in H^1(M,\R)$), 
such that, 
$ u \in \m O(c) $ implies that the perturbed lagrangian $ L- u$  has a unique c-minimizing measure, 
and he  asked  the question about  the  existence of  a residual set $\m O $ in $ C^\infty (M) $  such that any    lagrangian $ L - u$,
with  $u \in \m O$,  
has a unique c-minimizing measure for each  $c \in H^1(M,\R)$ .
In that direction, 
 Bernard and  Contreras in \cite{Bernard:2008b} obtained a residual  $\m O \subset C^\infty (M) $, 
such that any lagrangian $ L - u$,
with $ u \in \m O$,  
has at most   $ 1+ \dim  H^1(M,\R) $  ergodic  c-minimizing measures for  every $c \in H^1(M,\R)$. 
 
\smallskip

In the case where $ L $ is a mechanical lagrangian, that is 
$$ L (x, v) = \frac{1}{2} g_x (v, v) -u (x),$$ 
 it follows from  Carneiro's theorem (see \cite[Theorem 1] {Carneiro:1995}), that  the support of a c-minimizing measure of $ L $ is contained in the level of value  $ k =  \alpha(c) $ of the energy function $E_L=\partial_v L \cdot v-L$.
Applying    the Maupertuis' principle (see for example \cite[pg 247]{Arnold:1989})  to the mechanical lagrangian  $ L $,  we have that
the restriction of the Euler-Lagrange flow of  $ L $ to the energy level $E_{L}^{-1}(k)$, for   
$$ k > \min\{ \alpha (c) : c\in H^1(M,\R)\} =\alpha(0)= \max\{ u(x): x\in M\},$$   is a reparametrization of a  new geodesic flow for  a metric   conformally equivalent  to $ g$, but its corresponding   conformal factor  obligatorily changes when we change  the  energy  levels of $ k$.  

\smallskip

So, fixing  a  cohomology   $ c \in H^1(M,\R)$, from  the Carneiro-Maupertuis arguments  applied  to the family of mechanical lagrangians   obtained   through  perturbations of $ L $ by  summing   potentials in  the    residual subset  given by Mañé's  theorem, it follows that we can  obtain a residual set $ \m G(c) \in C^\infty(M)$,  such that, if $ f \in  \m G(c)$  then the   lagrangian   $\tilde{L}(x,v)= e^{f(x)} L(x,v) $  has a unique $c$-minimizing measure. 
 Note that,  if we want  to   study a property of  c-minimizing measures for a mechanical lagrangian,  that is independent of  $ c \in H^1(M,\R)$, we necessarily need to consider  an uncountable  set of energy values.  Thus, we will not be able to obtain a residual set of  metrics conformally equivalent to $ g$  using the arguments of   Carnerio-Maupertuis type on  the  generic family of mechanical lagrangians given  by the Bernard-Contreras' theorem.

 \bigskip 
 
 We will prove a analogous result of  \cite{Bernard:2008b} for the class of riemannian lagrangian by defining  an appropriate setting where we are able to apply an abstract result, proved in \cite{Bernard:2008b}, and to show a separability property (see Lemma~\ref{injective}) for  non-trivial   cohomology classes.
Hence we will  prove the following theorem. 

\begin{thm}
\label{Th. unique  miminizing measure} \  
Let  $(M,g)$ be a closed riemannian manifold with   $ \dim(H_1(M,\R)) \geq 1 $. 
There exists a residual subset $ \m G \ \subset C^\infty(M)$,  such that, 
if $$  \tilde{L}(x,v)=\frac{ 1}{2}\  e^{f(x)}\ g_x(v,v) $$
with  $ f \in \m G$,
 then for each  non-trivial cohomology  class $c \in H^1(M,\R)$ there exist at most 
 $ 1 + \dim H^1( M,\R)$ ergodic  c-minimizing measures 
for  $\tilde L$.
\end{thm}

 In Mather's studies on the dynamics of Tonelli lagrangian systems and the existence of Arnold diffusion, he found that understanding certain aspects of  the Aubry sets seems to help in the construction of orbits with interesting behavior.
  In particular, it seems useful to know that what we call  the quotient Aubry  set for cohomology $ c$,  and denote by $ (\overline {\m A}_L (c), \overline {\delta}_c) $,   is a totally disconnected metric space, that is, every connected component consists of a single point.  
  In \cite{Mather:2004}, Mather showed examples of mechanical lagrangians on $ T\mathbb{T}^d$, with $ d \geq2$ and the potential $u \in  C^{2d-3,1-\epsilon}(\mathbb{T}^d)$, whose  a  quotient Aubry set is isometric to an closed interval.  
  In \cite{Sorrentino:2007}, Sorrentino proved that Mather's counterexamples   are optimal,  he proved that, for every mechanical lagrangian on a closed manifold with a sufficiently regular potential, the quotient Aubry set corresponding to the zero cohomology class is totally disconnected. 
  Analogous result had been proved independently  by Fathi, Figalli and Rifford in \cite{FFR:2009}. Note that, the Bernard-Contreras' theorem imply that for generic lagrangian, in the Mañé sense, the quotient Aubry set $ (\overline {\m A}_L (c), \overline {\delta}_c) $ has at most a finite number of elements, for every $ c \in H^1 (M,\R) $. 
   This follows from the  fact that the elements of the quotient Aubry set,  that are usually  called (projected) static classes, are disjoint subsets and  support at least one ergodic minimizing measure (see \cite{Contreras:1997}).  
   
   The Arnold diffusion question is usually stated as perturbation with time periodic 
   potentials and looking for orbits which ``visit'' prescribed homology classes and
   whose energy goes to infinity. But the question still holds for autonomous or 
   riemannian lagrangians if the diffusion is restricted to a fixed energy level. 
   A result of Contreras-Paternain \cite{CP}, obtaining heteroclinics between static
   classes may be useful in that context. The hypothesis in \cite{CP} is, precisely,
   the finitness of the quotient Aubry set.

In the case  of lagrangians given by the kinetic energies  of riemannian metrics on  M, it is easy to see that, for the zero cohomology class the associated (projected) Aubry set  is $ M$ and the quotient Aubry set is a single point.  
Mather proved in \cite[Prop. 2]{Mather:2003} that, if $\dim M \leq 3 $ and  $ c \in H^1(M,\R)$
 then $ (\overline {\m A}_L (c), \overline {\delta}_c) $ is totally disconnected.
 Then, by Theorem~\ref{Th. unique  miminizing measure} and the above observations, we have the following corollary. 
 
\smallskip

\begin{cor}
\label{finite quotient Aubry set}
Let  $(M,g)$ be a closed Riemannian Manifold. 
There exists a residual subset $ \m G \ \subset C^\infty(M)$,  such that, if 
 $$  \tilde{L}(x,v)=\frac{ 1}{2}\  e^{f(x)}\ g_x(v,v) $$
with  $ f \in \m G$,
 then, for any $ c \in H^1(M,\R)$,  the quotient Aubry set  $ (\overline {\m A}_{\tilde L} (c), \overline {\delta}_c) $ has a finite number of elements.
\end{cor}
 
\section{Preliminary results}

\subsection{ The Aubry-Mather theory for  Tonelli lagrangian systems}
\label{Review of the Mather-Aubry Theory}
 
Let us  recall the   concepts of the Mather  and Aubry sets   for  (autonomous)  Tonelli  lagrangian introduced by J. Mather  in \cite{Mather:1991}, \cite{Mather:1993} 
and Ma\~n\'e in \cite{Ma7} respectivelly.

Let $ L : T M\rightarrow \R $ be a smooth { \it Tonelli lagrangian} 
defined in the tangent bundle $ TM$ of a closed smooth riemannian manifold $(M,g) $, i.e.,  $L(x,v) $ satisfy the two conditions:
\begin{itemize}
\item Convexity: for each fiber $ T_x M $, the restriction $
L(x,\cdot):T_xM \to \R $ has positive defined  Hessian,  and, 
\item  Superlinearity : 
$\lim_{\| v\|\rightarrow \infty} \frac{L(x,v)}{\|v\|} = \infty ,$
uniformly  in $ x \in M$. 

\end{itemize}
The  action of $ L $ over an
absolutely continuous curve $ \gamma :[a,b]\rightarrow M $ is
defined by:
$$
A_L(\gamma) = \int_a^b L(\gamma(t),\dot \gamma(t))\ dt.$$
The  extremal  curves  for the action are given by  solutions of the
{\it Euler-Lagrange equation} that in local coordinates can be
written as:
\begin{equation}
\label{E-L}
\frac{\partial L}{ \partial x}  - \frac{d}{dt }  \frac{\partial L}{\partial v} =0 .
\end{equation}
 Since $ L $ is
convex and $ M $ is compact,  the Euler-Lagrange equation
defines   a complete flow $ \phi_t: T M \rightarrow T M $, that is  called    {\it lagrangian flow} of $L $ and is defined by
$$
\phi_t ( x, v) = ( \gamma(t), \dot{\gamma}(t)), $$
where $ \gamma: \R \rightarrow M $ is the solution of (\ref{E-L}) with initial
conditions $ \gamma(0)= x $ and $ \dot{\gamma}(0) = v$.

The  {\it energy function} $  E_L : TM \rightarrow \R $ is defined as
\begin{equation}
\label{energy function}
 E_{L}(x,v):=\frac{\partial  L}{\partial v} (x,v)\cdot v -L(x,v)
\end{equation}
Any non-empty  energy level  $ E_L^{-1}(k) \subset TM   $ 
is compact and  invariant by the lagrangian flow.

\smallskip

We denote by $ \mathfrak B(L)$  the set of  all  Borel  probability measures   that are invariant by the  Euler-Lagrange flow of $ L$. 
Given a closed 1-form $ \omega $ on $ M$, 
consider the deformed lagrangian  $ L_\omega: TM \to \R$ defined by 
$$  L_\omega (x,v)= L(x,v) -  \omega_x(v) . $$
Since  $ d \omega =0$, 
the  lagrangians  $L_\omega$ and $ L$ have the same Euler Lagrange flow.
We say that  $ \mu \in \mathfrak B(L)$  is  a  {\it c-minimizing measure} of $ L$ 
if 
\begin{equation}
\label{minimizing measure}
 \int_{TM} L-\omega \ d\mu = \min\left\{ 
 \int_{TM} L-\omega\ d\nu
  :  \nu \in \mathfrak B(L)\right\},
\end{equation}
where $  c =  [\omega] \in H^1(M,\R)$.  
Let $\M_L(c) \subset \mathfrak B(L) $ be the set of all c-minimizing measures 
 (it only depends on the cohomology class $c$).  
The ergodic components  of a c-minimizing measure are also  c-minimizing measures,
so the set $\M_L(c) $ is a simplex whose extremal measures are ergodic c-minimizing measures.  The {\it$ \alpha$-function}, defined  as 
\begin{equation}
\label{alpha function}
 \alpha(c)=
-  \min\left\{ 
 \int_{TM} L-\omega\ d\nu
  :  \nu \in \mathfrak B(L)\right\},  
\end{equation}
 is a convex and superlinear function $ \alpha:H^1(M,\R) \to \R $. For each $ c \in  H^1(M,\R)$, we define the {\it Mather set  of cohomology class $ c$} as: 
$$  \tilde{\m M_L}(c)= \overline{\bigcup_{\mu \in \M_L (c)} \mbox{Supp}( \mu )} .$$
We set $ p(\tilde{\m M_L}(c))={\m M_L}(c)$, and call it the {\it projected Mather set}, where $ p:TM \to M $ denotes the canonical projection. The celebrated {\it graph theorem} proved by Mather  in \cite{Mather:1991}, asserts that $ \tilde{\m M}(c) $ is non-empty, compact,  invariant by the Euler-Lagrange flow  and $ p|_{\tilde{\m M_L}(c)}:   \tilde{\m M_L}(c) \to  {\m M_L}(c) $ is a bi-Lipschitz homeomorphism. 
In \cite{Carneiro:1995}  M. J. Carneiro proved (in the autonomous case), 
that  this set is contained in the energy level $ \m E(c):= E_L^{-1}(\alpha(c))$.
   
   Following Mather in \cite{Mather:1993}, for $ t>0$ and $ x,y \in M$, define 
 the {\it  action potential} for the  lagrangian deformed by a closed  1-forma $\omega $ as: 
 $$  h_\omega(x,y,t)=\inf \left\{   \int_0^t L(\gamma(s),\dot \gamma(s)) - \omega_{\gamma(s)} (\dot \gamma(s))  \ ds \right\}, $$
 where the infimum is taken over all absolutely  continuous curves $  \gamma:[0,t]\to M $ such that $ \gamma(0)=x$ and $ \gamma(t)=y$. The infimum is in fact a minimum 
 by  Tonelli's theorem. 

We define the {\it Peierls barrier} for the  lagrangian $ L-\omega $ as the function $ h_\omega: M \times M \to \R$ given by: 
\begin{equation*}
\label{Peiers barrier}
h_\omega(x,y) = \liminf_{t \to +\infty} \left\{ h_\omega (x,y,t) + \alpha([\omega]) t\right\}.
\end{equation*}
Define the {\it projected Aubry set for the cohomology class} $ c=[\omega]\in H^1(M,\R)$ as the set
\begin{equation*}
\label{projectet aubry set}
\m A_L(c)=\left\{ \ x \ \in M: \ h_\omega(x,y) = 0 \ \right\}.
\end{equation*}
By symmetrizing $ h_\omega $, we define the  {\it semidistance}   $\delta_c $ on $ \m A_L(c)$:
\begin{equation*}
\label{Mather distance }
\delta_c(x,y)=  h_\omega(x,y)+ h_\omega(y,x).
\end{equation*}
This function $\delta_c$ is non-negative and satisfies the triangle inequality. 

Finally, we define the {\it quotient Aubry set of the cohomology class }$ c=[\omega]\in H^1(M,\R)$, denoted by $( \overline{\m A_L}(c), \overline{\delta_c})$,  to be the metric space obtained by identifying two points $ x,y \in \m A_L(c) $ if their semidistance $\delta_c(x,y)$ vanishes. 

\bigskip

\subsection{Abstract  Results} 
In order  to state the  abstract theorem proved by  Bernard and Contreras  in \cite[Theorem~5]{Bernard:2008b}, we need to be given:  
\begin{itemize}
\item three topological vector spaces $ E, F, G$,
\item a continuous linear map $ \pi: F \to G $,
\item a bi-linear pairing, 
$ \langle \cdot , \cdot \rangle: E \times G \to \R$.
\item two metrizable convex compact subsets $ H \subset F$ and $ K \subset G $, 
such that $ \pi ( H) \subset K$.
\end{itemize} 
And we suppose that:
\begin{itemize}
\item[(i)]The map $ E \times K  \ni (u, \nu) \mapsto \langle u, \nu \rangle $ is continuous. 
\item[(ii)] The compact $ K $ is separated by $ E$. This means that, if $ \mu $ and $\nu $ are two different elements of  $K$, then there exists  $ u \in E$ such that $ \langle u, \mu \rangle \not= \langle u ,\nu \rangle$. 
\item[(iii)] $E $ is a Frechet    space. It means that $E$ is a topological vector space whose  topology is defined by a translation invariant metric, and that $ E $ is complete for this metric.  
\end{itemize}
So, given a linear functional $ L: H \to \R $, we denote by
$$ M_H(L) := \left\{\mu \in H : L(\mu) = \min_{\nu\in H} L(\nu) \right\}$$ 
and by $ M_K(L) $ the image $ \pi(M_H(L))$. These are compact convex subsets of $ 
H$ and $K$.

In this setting  Bernard 
and Contreras    proved the following abstract result: 
\begin{thm}[Th.~5 in \cite{Bernard:2008b}]
\label{abstract result}
For every finite dimensional affine subspace $ A \in H^* $, there exists a residual subset $ \m O(A) \subset E $ such that, for all $ u \in \m O(A) $ and all $ L \in A$, we  have that
$$ \dim M_K(L-u) \leq  \dim A.$$
\end{thm}

  \subsection{Holonomic measures}
  \label{holonomic measures}
Let $ \m C_q$ be the set of continuous  functions $ f : TM \to \R $ 
with growth at most quadratic, i.e. 
$$ \| f\|_q:= \sup_{(x,v)\in TM } \frac{|f(x,v)|}{1+g_x(v,v)^2 } < \infty .$$
endowed with the  norm defined above.
 Let $\m C^*_q$ be the dual space of $\m C_q$ endowed with the weak* topology. If $u\in C^0(M)$ then $u$ is bounded and hence the function 
$(x,v)\mapsto u(x)\, g_x(v,v)$ is in $\m C_q$. Therefore the function 
$\m C_q^*\to\R$ defined by
\begin{equation}\label{pigcont}
 \mu \longmapsto \int u(x)\, g_x(v,v) \, d\mu
\end{equation}
is continuous.

For each $C^1$ curve $ \gamma: \R \to M $  periodic,   with period $  T> 0 $,  we define the probability   
$ \mu_\gamma  $ on the Borel $\sigma$-algebra  of $TM $ by
 $$ \int \ f \ \mu_\gamma  = \frac{1}{T}\int_0^T f(\gamma(t),\dot \gamma(t)) dt , $$ for 
 all continuous functions $ f : TM\to \R $ with compact support.
Let $ \Gamma$  be  the set of probabilities of this form. 
Observe that $\dot\gamma(t)$ is bounded and hence 
 $ \Gamma $ is naturally embedded in $\m C_q^*$.  
Finally,  the set  of   {\it $L^2$ holonomic probabilities} $\m H_2 $  is the closure of $ \Gamma $ in 
$\m C_q^*$.  

\smallskip
 
Let $ L:TM \to M $ be a Tonelli lagrangian and $ \omega$ be a closed 1-form on $ M$. 
It follows from  Birkhoff Theorem (cf. \cite[prop. 1.1(b)]{Mane:1996a}) and Carneiro's result (cf. \cite[cor. 1]{Carneiro:1995}), that  
$ \mathfrak M_L ([\omega]) \subset  \m H_2$.

Ma\~n\'e defines the set of $L^1$ holonomic measures $\m H_1$
 in  the same fashion as above but using the norm
 $$
  \| f\|_\ell:= \sup_{(x,v)\in TM } \frac{|f(x,v)|}{1+\sqrt{g_x(v,v)} } < \infty .
 $$
 In   \cite[prop. 1.2 and 1.3]{Mane:1996a})
Mañé proved the existence of measures $ \mu \in \m H _1$ such that:
$$ \int_{TM} L - \omega \ d\mu = \min\left\{\int_{TM} L- \omega  \ d\nu : \nu \in \m H_1 \right\} $$
and that these measures are invariant by the Euler-Lagrange flow. Thus they are in 
$\mathfrak M_L([\omega]) \subset\m H_2$. Therefore,  the concept of c-minimizing measures    can be reformulated by taking the minimum of the action on the set of the  $ \m L^2$ holonomic probabilities. It is  
\begin{equation}
\label{Mather set}
 \mathfrak M_L(c) = \left\{ \mu \in \m H_2 : \int_{TM} L-\omega \ d\mu = \min_{\nu \in \m H} \int_{TM} L-\omega \ d\nu \right\} .
\end{equation}  
 The important goal is that the set  $\m H_2$  does not depend on the lagrangian.

\bigskip

\section{On the minimizing measures of geodesic flows}

We will apply  Theorem~\ref{abstract result} in the family of  lagrangians given by  all kinetic energies corresponding to all $C^\infty$-rienannian metrics on $ M$, that are conformally equivalent  to $g$.

 We consider the sets $ \m C_q, \m C_q^* $ and $ \m H_2$, as in the 
 section~\ref{holonomic measures}, and  we define the following setting: 
\begin{itemize}
\item  $ E = (C^\infty(M),d)  $ the metric space, where 
$$ d( u,v ) = \|u-v\|_\infty = \sum_{k\in\N} \frac{1}{2^k} \arctan ( \| u-v\|_{C^k} ).
$$ 
This metric space is complete, hence any residual subset of E is dense.

\item 
$F = \m C_q^* $ the vector space of 
continuous linear functionals $\mu:\m C_q\to\R$ endowed
with the weak* topology. Observe that $F=\m C_q^*$ is the set of 
finite Borel signed measures on $ TM $ such that  
$$ \int_{TM} g(v,v)\  d\mu < \infty. $$
We have that $\m H_2\subset F$.

\item $ G= (C^0(M))^* $ the vector space of continuous linear functions $\nu: C^0(M) \to\R$ endowed
with the weak* topology.

\item  The bilinear pairing $ \langle \cdot, \cdot \rangle: E \times G \to \R $ is defined by $  \langle u, \nu \rangle= \nu(u)$.

\item The  linear map $  \pi_g: F \to G $ is given by   
\begin{eqnarray*}
 \pi_g (\mu) :  C^0(M) &\longrightarrow & \R \\
 u &\longmapsto & 
  \pi_g (\mu)(u)= \int_{TM} u(x)\  g_x(v,v) \ d\mu .
\end{eqnarray*}
Note that 
$$|\langle u, \pi_g (\mu)\rangle| \leq \| u \|_{C^0} \int_{TM} g_x(v,v)\  d\mu ,$$ then
 $\pi_g(\mu)$ is continuous and hence $\pi_g(\mu)\in G$. 
 The argument in~\eqref{pigcont} 
 shows that for any $u\in E$, the function $F\ni \mu\mapsto \langle u, \pi_g (\mu)\rangle$
 is continuous. Therefore the map $\pi_g$ is continuous in the weak* topology of 
 $G$. Also the bilinearity implies that the map 
 $(u,\mu)\mapsto \langle  u, \pi_g (\mu)\rangle$ is continuous.

\item $n \in \N.$ 

\item The compact $ H_n \subset \m H_2 \subset F $ is the set of holonomic probability measures which are supported on the compact $ B_n=\{ (x,v) \in TM : g_x (v,v) \leq 2n\}.$ 

\item The compact $ K_n=\pi_g(H_n) \subset G $.  By duality we have that $ K_n $ is separated  by $ E $.
\end{itemize}   
 
 \smallskip
 
 Let $ L:TM \to \R $ be the riemannian lagrangian corresponding to $g$
and let $ \omega $ be  a closed 1-form on $ M $.   
Then $ L-\omega $ defines a   continuous linear functional  
$ A_c: \m H_2\to \R  $
 given by integration, it is  $$\nu \mapsto 
 A_c(\nu) =  \int_{TM}  L - \omega \  d\nu ,$$ 
where  $ [\omega]=c \in H^1(M,\R) $. 

Let $ A^n_{c,u}: H_n \to \R $ be the action defined as
$$ A^n_{c,u}(\mu)=
A_c(\nu)  + \langle  u, \pi_g (\mu)\rangle = \int_{TM} (L-\omega) +  \pi^*_g(u) \ d\mu $$
 where $  \pi^*_g(u)(x,v):=   u(x)\  g_x(v,v)$. We   denote by $ M_{H_n}(c,u) $  the set of measures $ \mu \in  H_n $ which  
 minimize the action  $ A^n_{c,u}: H_n \to \R $.

Let $ A $ be the affine subspace of continuous linear functionals on $\m H_2$ of the kind  
$ A_c$, 
with $ c\in H^1(M,\R)$. 
By applying  Theorem \ref{abstract result}, we obtain a residual 
subset $ \m O(A,n) \subset E $ such  that  
$ A_c  \in A  $ and $ u \in  \m O(A,n) $ 
imply that
$ \dim \pi_g (  M_{H_n}( c,u))  \leq \dim A.$

We set $$ \m O(A) = \bigcap_{n\in \N} \m O(A,n) .$$ By the Baire 
property $ \m O(A) $ is residual subset of $ E$ and we have that, if  
$ 
 A_c \in A , \ n\in\N $ and $ u \in  \m O(A) $, then:
$$ \dim \pi_g (  M_{H_n}( c,u))  \leq \dim A.$$

\bigskip 

Note that, if $ u(x) > -1/2 $ for all $ x \in M$,  then $\tilde L(u):=L +  \pi^*_g(u) $ is the   lagrangian corresponding to kinetic energy for the perturbed  riemannian metric $$ \tilde{g}(u):=  e^{\ln(1+2u)} g = (1+2u) g.$$ 
 Then, given a closed 1-form $ \omega $,
the lagrangian flow of  $ (L- \omega ) + \pi^*_g(u)= \tilde{L}(u) -\omega$  is the geodesic flow of the  metric   $ \g(u)$.
  
\smallskip
 
We recall that  by the Mather's graph theorem \cite{Mather:1991}, the union of the supports of   all
 $c$-minimizing measures of a Tonelli lagrangian  is a  compact set.  Therefore, by 
 (\ref{Mather set}), there   exists $\ 
m=m(g,u,c) \in \N $ such that
$$  \mathfrak M_{\tilde L(u)}(c)= M_{H_n}(c,u), \ \ \forall n \geq m  .$$

\smallskip
 
We  need  to compare  the dimension of the  sets $ \mathfrak M_{\tilde{L}(u)} (c) $ of minimizing measures for the 
Tonelli lagrangian 
$\tilde L(u) -\omega$  (with $[\omega]=c$)  with the dimension of the sets  $\pi_g (  
M_{H_n}(c,u))$, for 
$ A_c \in A$  and $ u \in \m G(A):= \m O(A)\cap \{ \ u \in C^\infty(M): \  u(x) > -1/2 \} $. 

\smallskip

 For the zero cohomology class, it is easy to see that  
  $ \mathfrak M_{L}(0) $ 
contain  all the  Dirac measures supported in a point of the zero section of $ TM $, 
i.e 
$ \delta_{(x,0)} \in \mathfrak  M_{L}(0)$ for all  $ x \in M$. On the other hand,  for each 
measure $ \delta_{(x,0)} \in \mathfrak M_{ L}(0)$,  by the definition of the map  
$ \pi_g $, for any $u \in E $  we have 
$$ \langle u,\pi_g (\delta_{(x,0)})\rangle = \int_{TM} u(x)\  g_x(v,v) \ d\delta_{(x,0)} =u(x)g(0,0)=0.$$
Then  $ \pi_g (\mathfrak M( L)) = \pi_g ( M_{H_n}(0,u)) = \{ 0\}\subset G=C^0(M)^*  $. 

 \smallskip
  
For any non-trivial cohomology class, we prove the  following two lemmas that  complete the proof of the Theorem~\ref{Th. unique  miminizing measure}.

\begin{lem}
\label{alfa0}
For  $L(x,v)=\frac 12 \, g_x(v,v)$
we have that $\alpha (c)>0$ if $c\ne 0$.
\end{lem}
\begin{proof}
Let $c\ne  0$ be a non-trivial cohomology class and $\omega$ a closed 1-form with $[\omega]=c$.
Then there exists a $C^1$ curve $\gamma$ on $M$ such that 
$\oint_\gamma \omega \ne 0$. Parametrizing $\gamma$ in opposite direction if necessary,
 we can
assume that $\oint_\gamma \omega =-b<0$.
Since $\gamma$ is not a fixed point,
we have that $a=A_L(\gamma)>0$.
For $s>0$ define $\gamma_s(t):=\gamma(st)$.
Then
$$
A_{L-\omega}(\gamma_s) = s^2 A_L(\gamma) + s\oint_\gamma \omega
= a s^2 - bs <0 \qquad \text{for}\qquad 0<s<{\textstyle \frac ba}.
$$
Let $\overline s= \frac b{2a}$, then from \eqref{alpha function} we have that
$$
\alpha(c)\ge - \int( L-\omega)\, d\mu_{\gamma_{\overline s}}
=- A_L(\gamma_{\overline s}) + \oint_{\gamma_{\overline s}}\omega 
=-a \overline{s}^2+b\overline{s}>0.
$$
\end{proof}

\medskip

\begin{lem}
\label{injective}
Let $ L $ be the lagrangian of a riemannian metric $ g$. Let $c \in H^1(M,\R )$ be   
a non-trivial cohomology class.  Then the  linear map
$$ \pi_g|_{ \mathfrak M_L(c )} : \mathfrak M_L(c ) \longrightarrow G $$
is injective.
\end{lem}
\begin{proof}

We need to  show  that if  $ \mu \not= \nu \in  \mathfrak M_L(c) $, then  $ \pi_g(\mu): E\to \R $ 
and
  $ \pi_g(\nu): E\to \R $ are two different  linear functionals, i.e there exist an element $ u \in C^\infty(M) $ such that:
\begin{equation}
\label{injective proof}
 \langle u,\pi_g ( \mu)\rangle = \int_{TM}  \pi^*_g(u)  \  d \mu \not=  \int_{TM}  \pi^*_g(u)  \  d \nu =\langle u,\pi_g ( \nu) \rangle 
 \end{equation}
Let $ p: TM \to M $ be the canonical projection. 
Let us  recall that the Mather set  $$\tilde {\m M_L}(c) = \bigcup_{ \mu \in \mathfrak M_L(c)} \supp(\mu) $$ is  compact  and 
is contained  in the energy level $$ \m E(c)= E_L^{-1}(\alpha(c)) =\{ (x,v) \in TM : g_x(v,v) = 2 \alpha(c)\  \},$$ see (\ref{energy function}) and (\ref{alpha function}). 
We also recall that,   by the graphic property \cite{Mather:1991}, the restriction    
$$ 
p|_{\tilde{\m M_L}(c)} : \tilde{\m M_L}(c) \to M 
$$ is injective and its inverse
 $ (p|_{\tilde{\m M_L}(c))})^{-1} $ is Lipschitz on the projected  Mather set $ \m M_L(c)$.  

Let $ \tilde B \subset TM $ be a Borel set such that $ \mu(\tilde B) \not= 
\nu (\tilde B) $.   Consider  $$ B := p (\tilde B   ) \subset M .$$ Observing that    $ \tilde B \subset p^{-1}(B) $,  we have $ \mu( \tilde B ) \leq \mu (p^{-1}(B)) $. Since $ \supp \mu \subset \m M_L(c) $, by the graphic property,
$$\mu\left( p^{-1} B \right) = \mu\left( p^{-1}( B) \cap \m M_L(c)\right) = \mu \left( (p\left|_{\m M_L(c)})^{-1} B \right.\right)  = \mu \left(\tilde B \cap \m M_L(c)\right) \leq \mu(\tilde B) .$$ 
Therefore $ \mu(\tilde{B}) = \mu\left( p^{-1}(B) \right)$. Repeating the same argument to the minimizing  measure $ \nu $,  we have that  $ \nu(\tilde{B}) = \nu\left( p^{-1}(B) \right)$.  Then
$$ \mu\left( p^{-1}(B) \right) \not= \nu\left( p^{-1}(B) \right) $$

Let $ \chi_B :M\to \R $ be  the characteristic function of the set 
$ B $. Then
$$\int_{TM} \chi_B (x) g_x(v,v) \ d \mu= \int_{p^{-1}(B)} g_x(v,v) \ d \mu =
\int_{p^{-1}(B)\cap \m M_L(c)}2  \alpha(c) \ d \mu = 2\alpha(c) \mu \left( p^{-1}(B)\right) 
$$
and similarly
$$\int_{TM} \chi_B (x) g_x(v,v) \ d \nu= \int_{p^{-1}(B)} g_x(v,v) \ d \nu =
\int_{p^{-1}(B)\cap \m M_L(c)} 2 \alpha(c) \ d \nu = 2\alpha(c) \nu \left( p^{-1}(B)\right)
$$

 Since by Lemma~\ref{alfa0}, $ \alpha(c)\not =0$,  we can choose 
  a $ C^\infty$-function $ u:M\to \R^+ $,
satisfying   
$$\int_{TM} u(x) g_x(v,v) \ d \mu \not=\int_{TM}  u(x) g_x(v,v) \ d \nu.$$
 This implies (\ref{injective proof}) and completes the proof of the lemma.

\end{proof}
    
\section*{Acknowledgments} J. A. G. Miranda  is very grateful to CIMAT for the hospitality an to CNPq-Brazil for the partial financial support.   
G. Contreras was partially supported by CONACYT, Mexico, grant 178838.
G. Contreras is very grateful to MSRI,  Berkeley, NFS Grant No. DMS-1440140 for the hospitality during Fall 2018 and to the Simons Foundation for financial support.

\bibliographystyle{amsalpha}

\begin{thebibliography}{FFR09}

\bibitem[Arn89]{Arnold:1989}
V.~I. Arnold, \emph{Mathematical methods of classical mechanics}, second ed.,
  Graduate Texts in Mathematics, vol.~60, Springer-Verlag, New York, 1989,
  Translated from the Russian by K. Vogtmann and A. Weinstein. \MR{997295}

\bibitem[BC08]{Bernard:2008b}
P. Bernard and G. Contreras, \emph{A generic property of families of
  {L}agrangian systems}, Ann. of Math. (2) \textbf{167} (2008), no.~3,
  1099--1108. \MR{2415395 (2009d:37113)}

\bibitem[Car95]{Carneiro:1995}
M.~J.~D. Carneiro, \emph{On minimizing measures of the action of autonomous
  {L}agrangians}, Nonlinearity \textbf{8} (1995), no.~6, 1077--1085.
  \MR{1363400 (96j:58062)}

\bibitem[CDI97]{Contreras:1997}
G. Contreras, J. Delgado, and R. Iturriaga, \emph{Lagrangian flows:
  the dynamics of globally minimizing orbits. {II}}, Bol. Soc. Brasil. Mat.
  (N.S.) \textbf{28} (1997), no.~2, 155--196. \MR{1479500 (98i:58093)}
  
\bibitem[CP02]{CP}
Gonzalo Contreras and Gabriel~P. Paternain, \emph{Connecting orbits between
  static classes for generic {L}agrangian systems}, Topology \textbf{41}
  (2002), no.~4, 645--666. \MR{1905833 (2003i:37059)}  

\bibitem[FFR09]{FFR:2009}
A. Fathi, A. Figalli, and L. Rifford, \emph{On the {H}ausdorff
  dimension of the {M}ather quotient}, Comm. Pure Appl. Math. \textbf{62}
  (2009), no.~4, 445--500. \MR{2492705}

\bibitem[Ma{\~n}95]{Ma7}
\bysame, \emph{Lagrangian flows: the dynamics of globally minimizing orbits},
  International Conference on Dynamical Systems (Montevideo, 1995), Longman,
  Harlow, 1996, Reprinted in Bol. Soc. Brasil. Mat. (N.S.) {\bf 28} (1997), no.
  2, 141--153., pp.~120--131. \MR{1479499 (98i:58092)}


\bibitem[Ma{\~n}96]{Mane:1996a}
R. Ma{\~n}{{\'e}}, \emph{Generic properties and problems of minimizing
  measures of {L}agrangian systems}, Nonlinearity \textbf{9} (1996), no.~2,
  273--310. \MR{1384478 (97d:58118)}
  
  

\bibitem[Mat91]{Mather:1991}
J.~N. Mather, \emph{Action minimizing invariant measures for positive
  definite {L}agrangian systems}, Math. Z. \textbf{207} (1991), no.~2,
  169--207. \MR{1109661 (92m:58048)}

\bibitem[Mat93]{Mather:1993}
\bysame, \emph{Variational construction of connecting orbits}, Ann. Inst.
  Fourier (Grenoble) \textbf{43} (1993), no.~5, 1349--1386. \MR{1275203
  (95c:58075)}

\bibitem[Mat03]{Mather:2003}
\bysame, \emph{Total disconnectedness of the quotient {A}ubry set in low
  dimensions}, Comm. Pure Appl. Math. \textbf{56} (2003), no.~8, 1178--1183,
  Dedicated to the memory of J\"urgen K. Moser. \MR{1989233}

\bibitem[Mat04]{Mather:2004}
\bysame, \emph{Examples of {A}ubry sets}, Ergodic Theory Dynam. Systems \textbf{24} (2004), no.~5, 1667--1723. \MR{2104599}

\bibitem[Sor08]{Sorrentino:2007}
A. Sorrentino, \emph{On the total disconnectedness of the quotient {A}ubry
  set}, Ergodic Theory Dynam. Systems \textbf{28} (2008), no.~1, 267--290.
  \MR{2380310}

\end{thebibliography}
\providecommand{\bysame}{\leavevmode\hbox to3em{\hrulefill}\thinspace}
\providecommand{\MR}{\relax\ifhmode\unskip\space\fi MR }
\providecommand{\MRhref}[2]{%
  \href{http://www.ams.org/mathscinet-getitem?mr=#1}{#2}
}
\providecommand{\href}[2]{#2}

\end{document}